\documentclass[twocolumn]{autart}    

\usepackage{graphicx}          
\let\theoremstyle\relax
   
\usepackage{amsmath}  
\usepackage{amsthm}   
\usepackage{amsfonts} %
\usepackage{algorithm}
\usepackage{algpseudocode}
\usepackage{subfigmat}
\usepackage{subfigure}
\usepackage{rotating}
\usepackage{enumerate}
\usepackage{adjustbox}

\def\bm#1{\mathchoice                             
  {\mbox{\boldmath$\displaystyle#1$}}%
  {\mbox{\boldmath$#1$}}%
  {\mbox{\boldmath$\scriptstyle#1$}}%
  {\mbox{\boldmath$\scriptscriptstyle#1$}}}

\def\bm#1{\mathchoice                             
  {\mbox{\boldmath$\displaystyle#1$}}%
  {\mbox{\boldmath$#1$}}%
  {\mbox{\boldmath$\scriptstyle#1$}}%
  {\mbox{\boldmath$\scriptscriptstyle#1$}}}

\newcommand{\Rbb}[1]{\mathbb{R}^{#1}}		    
\newcommand{\vect}[1]{\mathbf{#1}} 

\DeclareMathOperator{\diag}{diag}
\DeclareMathOperator{\vectorize}{vec}

\newtheorem*{rep@theorem}{\rep@title}
\newcommand{\newreptheorem}[2]{%
\newenvironment{rep#1}[1]{%
 \def\rep@title{#2 \ref{##1}}%
 \begin{rep@theorem}}%
 {\end{rep@theorem}}}
\makeatother
\newtheorem{remark}{Remark}
\newtheorem{theorem}{Theorem}
\newtheorem{problem}{Problem}
\newreptheorem{theorem}{Theorem}
\newtheorem{lemma}{Lemma}
\newreptheorem{lemma}{Lemma}

\usepackage{nicefrac}
\newcommand{\onehalf}{\nicefrac{1}{2}}

\usepackage{subfigure} 
\usepackage{subfigmat} 

\begin{document}

\begin{frontmatter}

\title{COSMIC: fast closed-form identification from large-scale data for LTV systems} 


\author[GMV,IST]{Maria Carvalho}\ead{mariaccarvalho@tecnico.ulisboa.pt},
\author[FCT]{Claudia Soares}\ead{claudia.soares@fct.unl.pt}, 
\author[GMV]{Pedro Lourenco}\ead{palourenco@gmv.com},
\author[ISR]{Rodrigo Ventura}\ead{rodrigo.ventura@isr.tecnico.ulisboa.pt}

\address[GMV]{GMV, Portugal}    
\address[IST]{Instituto Superior Técnico, U. de Lisboa, Portugal}        
\address[FCT]{NOVA LINCS, Computer Science Department, Nova School of Science and Technology, U. Nova de Lisboa, Portugal}
\address[ISR]{Institute for Systems and Robotics - Lisboa, Portugal}
          
\begin{keyword}                           
Closed-form; system identification; linear time-variant.  
\end{keyword}                             

\begin{abstract}                          
We introduce a closed-form method for identification of discrete-time linear time-variant systems from data, formulating the learning problem as a regularized least squares problem where the regularizer favors smooth solutions within a trajectory. We develop a closed-form algorithm with guarantees of optimality and with a complexity that increases linearly with the number of instants considered per trajectory. The COSMIC algorithm achieves the desired result even in the presence of large volumes of data. Our method solved the problem using two orders of magnitude less computational power than a general purpose convex solver and was about 3 times faster than a Stochastic Block Coordinate Descent especially designed method. Computational times of our method remained in the order of magnitude of the second even for 10k and 100k time instants, where the general purpose solver crashed.
To prove its applicability to real world systems, we test with spring-mass-damper system and use the estimated model to find the optimal control path. Our algorithm was applied to both a Low Fidelity and Functional Engineering Simulators for the Comet Interceptor mission, that requires precise pointing of the on-board cameras in a fast dynamics environment. Thus, this paper provides a fast alternative to classical system identification techniques for linear time-variant systems, while proving to be a solid base for applications in the Space industry and a step forward to the incorporation of algorithms that leverage data in such a safety-critical environment.
\end{abstract}

\end{frontmatter}

\section{Introduction}

The identification of a system's behavior constitutes a problem in many areas of engineering~\cite{aastrom1971system}, from biological relations~\cite{crampin2006system} to physical systems' dynamics~\cite{agbabian1991system}, and it is the first step to any control design problem. It is imperative to have a system model that correctly represents its interaction with the environment and allows for more accurate predictions of its response to external stimuli. 

Typically, this issue is addressed by considering previous knowledge about the system, applying first principles of physics and measuring whenever possible~\cite{Glad2013}. However, for more intricate systems, this approach fails because it may become impossible to measure all the parameters needed to characterize a behavior completely or the system is too complex to be modeled by simple equations~\cite{budiyono2011modeling}. In order to address these limitations, a data-driven approach to the system identification problem, where input-output data is used to find a model that describes a particular system, has become more common~\cite{brunton2019data}. Finding more efficient and comprehensive universal algorithms to deal with a large amount of data is essential to derive reasonable data-driven solutions and constitutes a pressing issue~\cite{cheng2015interplay}.

According to Lamnabhi-Lagarrigue et al.~\cite{lamnabhi2017systems}, system modeling represents a significant cost in complex engineering projects, sometimes up to 50\% of the total cost, mainly due to the man-hours of the expert engineers dedicated to this task. Thus, as it is also stated, it is essential to create practical system identification tools that adapt to a wide range of problems and achieve a solution in a time-constrained setting. Such tools can be especially useful in a Space mission project that represents huge cost efforts for entire countries and agencies. The integration of Machine Learning (ML), and Guidance, Navigation and Control (GNC) can be helpful in this set of problems, from  building robust control frameworks that address parameter varying systems to applying verification and validation techniques to a system in a more efficient way.


We are going to analyze the problem of pointing a spacecraft to a comet during a fly-by, inspired by the Comet Interceptor mission~\cite{esa:2019:AssessmentMissionIntercept}. The pointing requirement can be posed as an attitude guidance and control problem, aiming at keeping the comet within the field-of-view of the scientific instruments as shown in Figure~\ref{fig:variation}(a). The fly-by is performed at high speed from a large distance to the comet, which makes meeting the requirement harder the closer the spacecraft gets to the target - significant torques and angular velocities are required (see Figure~\ref{fig:variation}(b)). Considering the linearized error dynamics w.r.t. the reference attitude, angular velocity, and torque, it becomes clear that the evolution is driven by the reference angular velocity, thus also having relevant variations throughout the scenario. However, it can be observed that this variation between instants is bounded, when considering a discrete-time setting as shown in Figure~\ref{fig:variation}(c). For optimal controller design purposes, nonlinear systems such as the error dynamics can benefit from being modeled as linear time-variant (LTV)~\cite{schoukens2019nonlinear,vanbeylen2013nonlinear}.


\begin{figure}[!htb]
\begin{subfigmatrix}{3}
    \subfigure[Comet Interceptor Scenario.]
    {\includegraphics[width=0.85\linewidth]{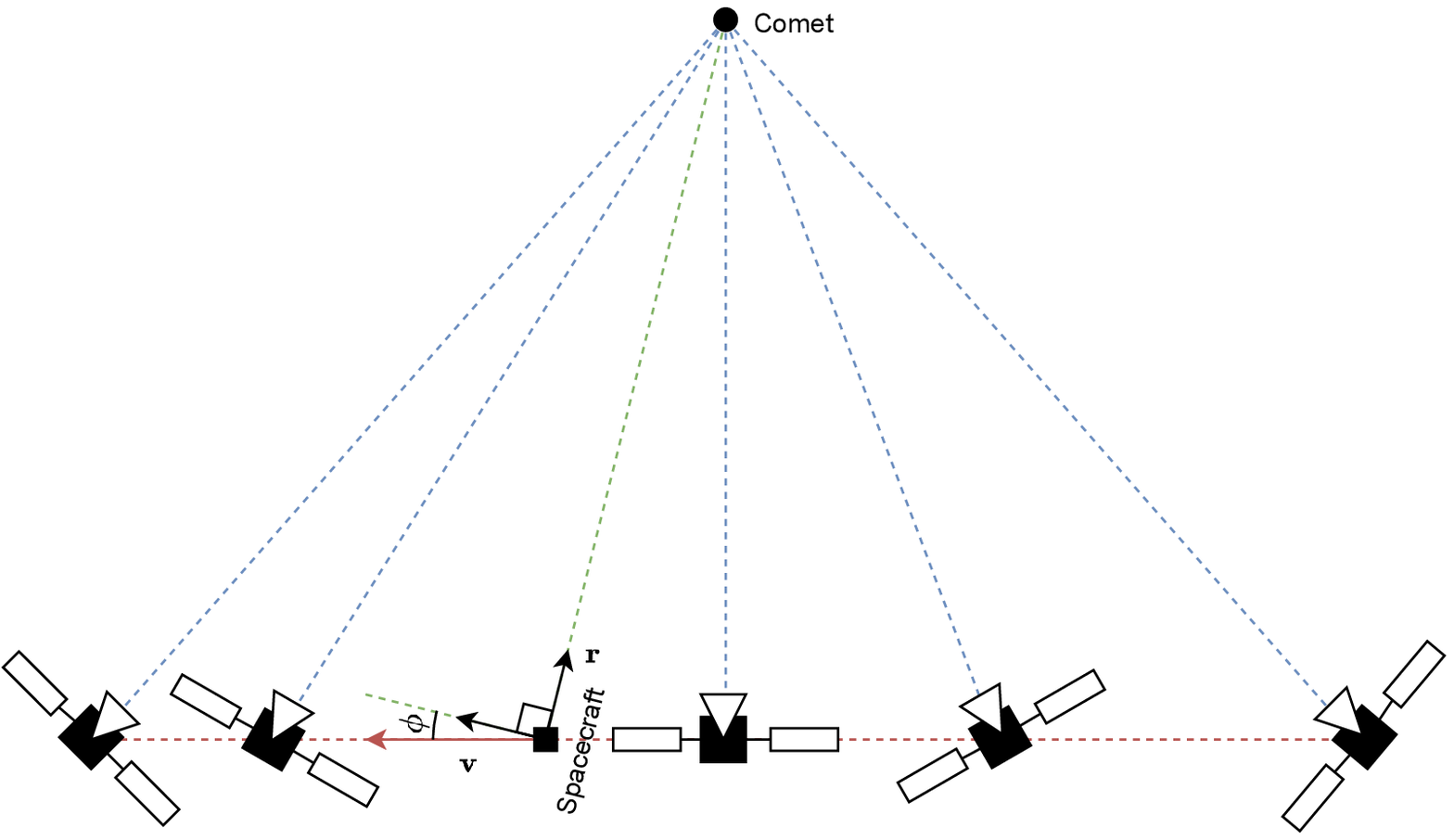}}
    \subfigure[Angular velocity evolution.]
    {\includegraphics[width=0.49\linewidth]{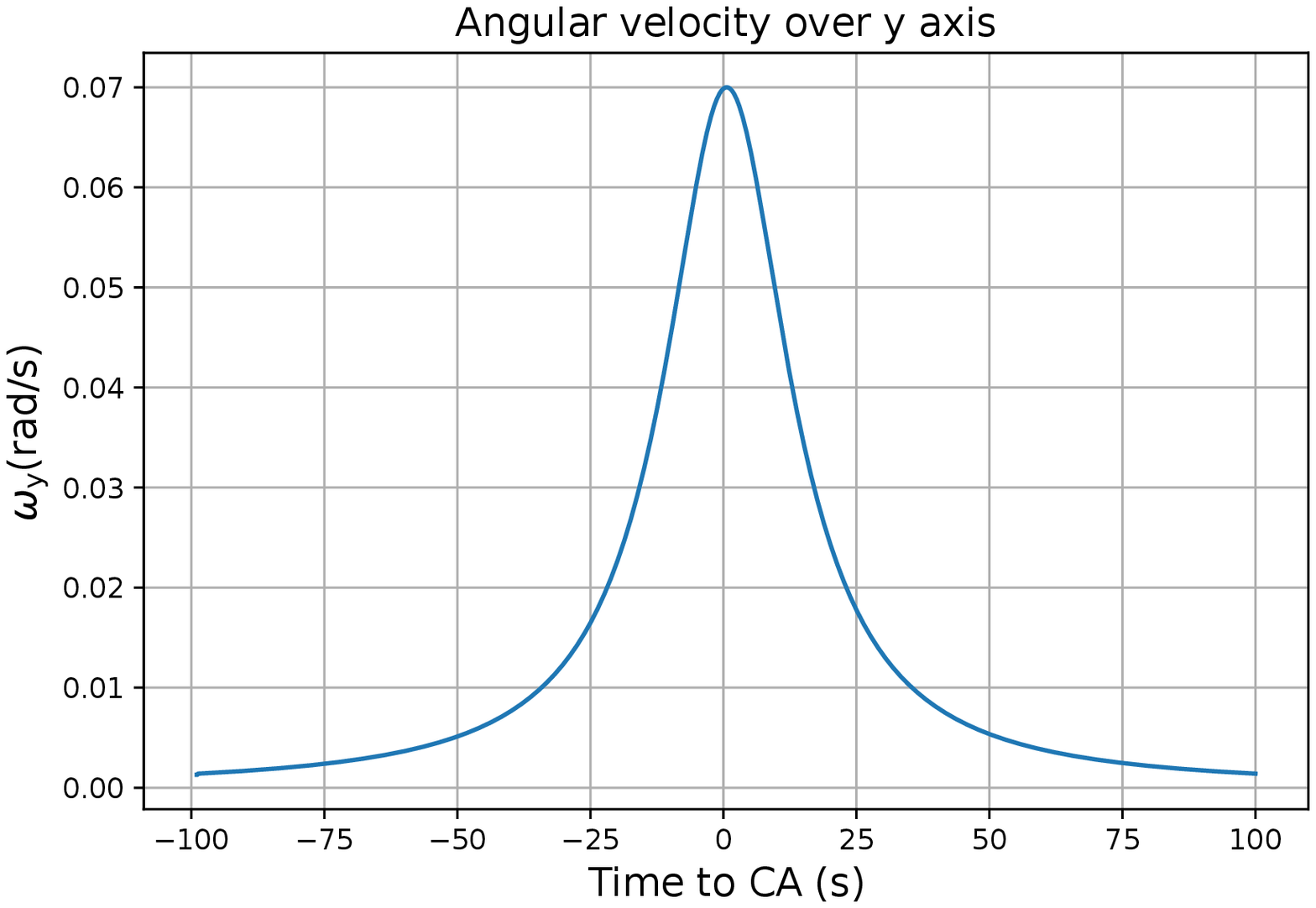}}
    \subfigure[Variation of error dynamics.]
    {\includegraphics[width=0.49\linewidth]{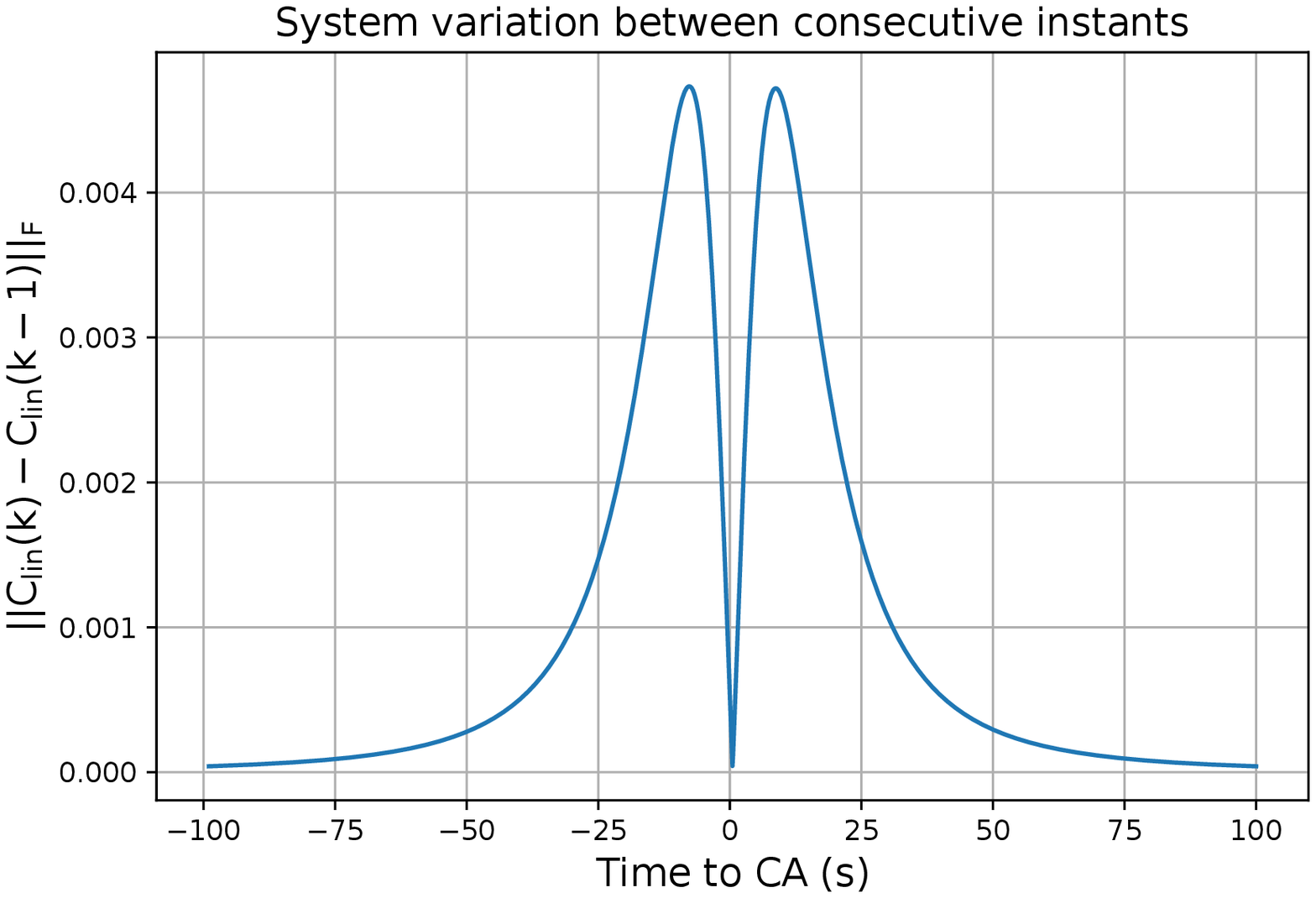}}
\end{subfigmatrix}
\caption{The Comet Interceptor mission: overall trajectory, associated angular velocity and variation of error dynamics. Due to the large distance and high speed, to point the spacecraft to the comet, the angular velocity has a large amplitude, with reaching significant values near the closest approach. When considering the associated linearized error dynamics, it is clear that the variations are bounded from one instant to the next, entailing smoother transitions.}
\label{fig:variation}
\end{figure}

\subsection{Related work}

A large body of work has been developed in order to incorporate diverse information in system identification and controller synthesis. The results described in~\cite{Ljung1998} establish the grounds for system identification from observed data. Earlier, the work of Kumpati et al.~\cite{kumpati1990identification} already proposed a system identification strategy based on Neural Networks. Nevertheless, data availability and computational and algorithmic tools were limited, so performance and speed were constrained.

Recently, there has been more focus on providing guarantees that a problem can be solved to a given precision in finite time, which is crucial to bridge the gap to optimal control. A lot of work has been developed regarding linear time-invariant (LTI) systems ~\cite{hardt2016gradient,dean2018sample,sarkar2019near}. Although these are significant achievements, not many real world systems can be represented as LTI and generalizing these findings to LTV systems is essential for a broader application of the results.

The literature is not as extensive regarding the direct application of data-driven system identification to linear time-variant systems. However, this problem has been a concern for many years. Dudul et al~\cite{dudul2004identification} apply Feedforward Neural Networks to identify an LTV system, assuming a transfer function characterization. Lin et al.~\cite{lin2020system} propose an episodic block model for an LTV system, where parameters within a block are kept constant, followed by the exploration of a meta-learning approach for system identification divided into two steps: an offline initialization process and online adaptation. Formetin et al~\cite{formentin2021control} have recently proposed a system identification procedure that takes control specifications into account in the form of regularization.

\subsection{Main contribution}
Given what is presented above, this work aims to develop system identification techniques for LTV systems in order to create models that are simultaneously accurate enough to represent a wide range of phenomena while being simple enough to use with well-known and widely used control techniques, leading to more robust control algorithms. We encode problem specific information in the formulation as a regularized least squares. Our main contribution is an algorithm that guarantees a solution of the linear time-variant system identification problem in a finite and deterministic number of operations. The main challenges of this proposal are the amount of data needed to characterize one trajectory due to its changing dynamics and guaranteeing a solution for the problem in finite-time. 



\section{Problem Statement}

We denote a discrete linear time-variant system transition equation as
\begin{equation} \label{eq:sys_def}
    \vect{x}(k+1) = \vect{A}(k)\vect{x}(k) + \vect{B}(k) \vect{u}(k),
\end{equation}
such that $k\in [0, ..., N-1]$, with $N+1$ being the total number of instants considered to be part of one trajectory. In a given instant $k$, the state is $\vect{x}(k)\in\Rbb{p}$ and the control input is $\vect{u}(k)\in\Rbb{q}$. System parameters $\vect{A}(k)\in\Rbb{p \times p}$ and $\vect{B}(k)\in\Rbb{p \times q}$ are, respectively, the dynamics and control matrices that define the system's response, and the unknown variables we aim to derive from data.

Assume we collect a dataset with $L$ trajectories for the $N+1$ time steps per trajectory. Taking into account all the data available, we need to additionally define the matrices that contain the state measurements, $\vect{X}(k) \in \Rbb{p \times L}$ as $\vect{X}(k) =  \begin{bmatrix}
    \vect{x}_1(k) & \vect{x}_2(k) & ... & \vect{x}_L(k)
    \end{bmatrix}$ and the control data, $\vect{U}(k) \in \Rbb{q \times L}$ as $\vect{U}(k) =  \begin{bmatrix}
    \vect{u}_1(k) & \vect{u}_2(k) & ... & \vect{u}_L(k)
    \end{bmatrix},$ of all the $L$ different trajectories for the $k$-th instant of the trajectory. Moreover, we define $\vect{X'}(k)$ as $\vect{X'}(k) = \vect{X}(k+1)$.

To find the proper solution for system~\eqref{eq:sys_def}, we start by defining the optimization variable $\vect{C}(k) \in \Rbb{(p+q) \times p}$ as $\vect{C} (k) =  \begin{bmatrix}
    \vect{A}^T (k) \\ \vect{B}^T (k) \end{bmatrix}$, with $\vect{C} = \left[\vect{C}(k)\right]_{k=0}^{k=N-1}$, such that $\vect{C} \in  \Rbb{N(p+q) \times p}$. 
For notational simplicity, we define
    $\vect{D}(k) \in \Rbb{L \times (p+q)}$ as $\vect{D} (k) =  \begin{bmatrix}
    \vect{X}^T (k) & \vect{U}^T (k) \end{bmatrix},$
with $\vect{D} = \left[\vect{D}(k)\right]_{k=0}^{k=N-1}$, following
$  \vect{V} = \diag(\vect{D} (k))$ with $\vect{V} \in \Rbb{NL \times N(p+q)}$.

To limit the variation of the system between instants, thus encoding the verified assumption that a system will not change drastically from one time step to the next (for example, the system described in Figure~\ref{fig:variation}), as addressed in the motivation of this work, we add a term to the cost function, with regularization parameters $\lambda_k > 0$. By allowing for the $\lambda_k$ to vary throughout the trajectory, we impose more flexibility to the problem formulation, covering a wider range of scenarios.

As such, we are solving
\begin{problem}\label{prob:prob-with-sum}
\begin{equation} 
\begin{aligned}
    \underset{\vect{C}}{\text{minimize}}
    \hspace{0.2cm} f(\vect{C}) & :=  \frac{1}{2} \| \vect{V} \vect{C} - \vect{X'}\|_F ^2 \\ & +   \frac{1}{2} \sum_{k=1}^{N-1} \| \lambda_k^{\onehalf} (\vect{C}(k) - \vect{C}(k-1))\|_F ^2.
\end{aligned}
\end{equation}
\end{problem}

\begin{remark}
The problem we are solving can be seen as a trade off between how close the optimization variable is to the data and how much we allow for it to change between instants. Thus, we can address the tuning of parameter $\lambda_k$ as the tuning of the relative importance of each objective. Higher values of $\lambda_k$ are congruent with little variation of the system behavior between instants and will emphasize the weight of the second term has in the cost function and lower values will allow for more drastic changes between $k-1$ and $k$.
\end{remark}
Moreover, to aid the analysis, we can define the second term of $f$ writing all the weighted difference equations $\lambda_k^{\onehalf}\left(\vect{C}(k) - \vect{C}(k-1)\right)$ as a matrix product 
\begin{equation} \label{eq:dif_equation}
\resizebox{.45\textwidth}{!}{$
  \underbrace{ \begin{bmatrix}
    \lambda_1^{\onehalf} & 0 & ... & 0 & 0\\
    0 & \lambda_2^{\onehalf} & ... & 0 & 0\\
    ... & ...  & ... & ... & ... \\
    0 & 0 & ... & \lambda_{N-2}^{\onehalf} & 0\\
    0 & 0 & ... & 0 & \lambda_{N-1}^{\onehalf} \\ 
 \end{bmatrix}}_\text{$\vect{\Upsilon}^{\onehalf}$}
 \underbrace{\begin{bmatrix} 
    -\vect{I} & \vect{I} & \vect{0} & ... & \vect{0} & \vect{0} & \vect{0}\\
    \vect{0} &-\vect{I} & \vect{I}  & ... & \vect{0} & \vect{0} & \vect{0}\\
    ... & ...  & ... & ... & ... & ... & ... \\
    \vect{0} &\vect{0} & \vect{0} & ... & -\vect{I} & \vect{I} & \vect{0}\\
    \vect{0} &\vect{0} & \vect{0} & ... & \vect{0} & -\vect{I} & \vect{I}\\
    \end{bmatrix}}_\text{$\vect{F}$}  
 \begin{bmatrix} 
    \vect{C}(0)\\
    \vect{C}(1)\\
    ...\\
    \vect{C}(N-1)\\
    \end{bmatrix},$}
\end{equation}
allowing Problem~\ref{prob:prob-with-sum} to be written as
\begin{problem} \label{prob:RLS}
\begin{equation} \label{eq:prob-compact}
    \underset{C}{\text{minimize}}
    \hspace{0.5cm} f(\vect{C}) :=    \frac{1}{2} \| \vect{V} \vect{C} - \vect{X'}\|_F ^2  +   \frac{1}{2}  \| \vect{\Upsilon}^{\onehalf} \vect{F}  \vect{C} \|_F^2.
\end{equation}
\end{problem}

This statement makes explicit that we are solving an unconstrained convex optimization problem. As Problem~\ref{prob:RLS} is an unconstrained minimization of a quadratic function, solving~\eqref{eq:prob-compact} amounts to solving equation
\begin{equation} \label{eq:cond}
    \vect{\nabla} f (\vect{C}) = \vect{0}.
\end{equation} 
The following linear equation solves the problem, 
\begin{equation} \label{eq:solution}
    \begin{split}
         \nabla f(\vect{C}) & = \vect{V}^T(\vect{V} \vect{C}^* - \vect{X'}) + \vect{F}^T \vect{\Upsilon} \vect{F}  \vect{C}^* = 0 \\
                     & \Leftrightarrow \left(\vect{V}^T \vect{V} + \vect{F}^T \vect{\Upsilon} \vect{F} \right) \vect{C}^* = \vect{V}^T \vect{X'} \\
                     & \Leftrightarrow \vect{C}^* = \left(\vect{V}^T \vect{V} + \vect{F}^T \vect{\Upsilon} \vect{F} \right) ^{-1} \vect{V}^T \vect{X'}
    \end{split}.
\end{equation}

We now derive a condition on data that guarantees that the information collected is enough to reach a solution and correctly identify the system.

\begin{reptheorem}{theor:unique}[\textbf{\textit{Informal}}: When is the dataset large enough?]

When the collected data is sufficiently varied, there is a unique solution to Problem~\ref{prob:prob-with-sum}.
Further, after each collected trajectory, it is possible to identify if there is enough information for attaining a unique solution by computing the sum of the required trajectories' covariances and testing it for positive definiteness.
\end{reptheorem}

To prove Theorem~\ref{theor:unique}, we will require two lemmas.

\begin{lemma}\label{lemma:im(b)}
Let $\vect{B} \succeq 0 $, $\vect{B}\in \Rbb{m \times m}$ and $\vect{x} \neq 0$. If $\vect{x} \in Im(\vect{B})$ then $\vect{x}^T \vect{B} \vect{x} > 0$.
\end{lemma}

\begin{proof}
If the result of the Lemma does not hold, we would have that $\vect{x} \in \text{Im}(\vect{B})$ and $\vect{x} \in \text{Ker}(\vect{B})$. However, since $\text{Im}(\vect{B}) \cap \text{Ker}(\vect{B}) = \emptyset$, the only option would be to have $\vect{x} = \vect{0}$, which violates the conditions of the Lemma, thus proving that the inference holds.
\end{proof}

\begin{lemma}\label{lemma:a+b}
The following two expressions are equivalent:~
\begin{enumerate}[(i)]
    \item \begin{enumerate}[(a)]
        \item $\vect{A} \succeq \vect{0}$, $\vect{B} \succeq \vect{0}$;
        \item $\forall \vect{v} \in \text{Ker}(\vect{B})\setminus \{\vect{0}\}: \hspace{0.5cm} \vect{v}^T \vect{A} \vect{v} > 0$;
    \end{enumerate}
    \item $\vect{A} + \vect{B} \succ \vect{0}$.
\end{enumerate}
\end{lemma}

\begin{proof}
The proof that (ii) implies (i) follows trivially from the fact that both $\vect{A}$ and $\vect{B}$ are positive semidefinite.

To prove the other way around, we first remember that if $\vect{B} \succ 0$, then $\text{Im}(\vect{B}) \perp \text{Ker}(\vect{B})$.

Let $\vect{v} = \vect{v_I} + \vect{v}_K$, $\vect{x}\neq \vect{0}$, with $\vect{v_I} \in \text{Im}(\vect{B})$ and $\vect{v_K} \in \text{Ker}(\vect{B})$.

If $\vect{v_I} \neq \vect{0}$, then
\begin{equation} \label{eq:lemma-cond}
\begin{aligned}
    & (\vect{v_I} + \vect{v}_K)^T (\vect{A}+\vect{B})(\vect{v_I} + \vect{v}_K) =\\ & (\vect{v_I} + \vect{v}_K)^T \vect{A}(\vect{v_I} + \vect{v}_K) +\\
    & (\vect{v_I} + \vect{v}_K)^T \vect{B}(\vect{v_I} + \vect{v}_K).
\end{aligned}
\end{equation} From the conditions of the Lemma, we know that the first term of the sum holds, as $\vect{v}^T \vect{A} \vect{v} > 0$. For the second term, we recall Lemma~\ref{lemma:im(b)}, yielding that $\vect{v}^T \vect{B}\vect{v} > 0$, thus we infer that the LHS product is greater than zero. From this, we get
\begin{equation*}
\vect{A} + \vect{B} \succ 0.
\end{equation*}

If $\vect{v_I} = \vect{0}$,~\eqref{eq:lemma-cond} is turned into
\begin{equation*}
    \vect{v}^T_K (\vect{A}+\vect{B}) \vect{v}_K = \vect{v}^T_K \vect{A} \vect{v}_K > \vect{0}.
\end{equation*} Thus, we prove that $\vect{A}+\vect{B} \succ \vect{0}$.
\end{proof}

\begin{theorem}\label{theor:unique}[Existence and uniqueness of the solution of Problem~\ref{prob:prob-with-sum} and the relation with the data]

Let the collected dataset $\vect{D}$ be comprised of $L$ trajectories drawn independently
\begin{equation*}
    \vect{D} = \{ (\vect{x}_\ell(k), \vect{u}_l(k)): k \in \{ 0,...,N-1\}, \ell \in \{1,...,L\}\}
\end{equation*}
and let $\vect{\Sigma}_\ell$ be the empirical covariance of trajectory $\ell$, i.e.,
\begin{equation}\label{eq:cov}
    \vect{\Sigma}_\ell = \frac{1}{N}\sum_{k=0}^{N} \begin{bmatrix}
        \vect{x}_\ell(k) \\ \vect{u}_\ell(k)  
    \end{bmatrix} \begin{bmatrix}
        \vect{x}_\ell(k)  \\ \vect{u}_\ell (k) 
    \end{bmatrix}^T.
\end{equation}

Then, the following statements are equivalent:
\begin{enumerate}[(i)]
    \item the solution of Problem~\ref{prob:prob-with-sum} exists and is unique;
    \item the empirical covariance of the data, $\vect{\Sigma} = \vect{\Sigma}_1 + \vect{\Sigma}_2 + ... + \vect{\Sigma}_L$, as in~\eqref{eq:cov} is positive definite.
\end{enumerate}
\end{theorem}
\vspace{-1cm}
\begin{proof}
Let us define $\Tilde{\vect{x}}(k) = \vectorize(\vect{X}(k))$, $\Tilde{\vect{c}}(k) = \vectorize(\vect{C}(k))$ and the new cost function
\begin{equation*}\label{eq:prob-vec}
    \begin{aligned}
        \hspace{0.5cm} f(\Tilde{\vect{c}}) :=  & \frac{1}{2} \sum_{k=0}^{N-1} \| \Tilde{\vect{x}}(k+1) - (\vect{D}(k) \otimes \vect{I}) \Tilde{\vect{c}}(k)\|^2 + \\ & \frac{1}{2} \sum_{k=1}^{N-1} \| \lambda_k^{\onehalf} (\Tilde{\vect{c}}(k) - \Tilde{\vect{c}}(k-1))\|^2.
    \end{aligned}
\end{equation*} Then, Problem~\ref{prob:prob-with-sum} can be written as
\begin{equation*}
    \underset{\Bar{\vect{c}}}{\text{minimize}}
    \hspace{0.5cm} f(\Tilde{\vect{c}}).
\end{equation*}
The proof follows by verifying that the second-order condition for convex minimization problem from~\cite{boyd2004convex} holds. This states that a function $f$ is strictly convex if and only if its domain is convex and its Hessian is positive definite. The Hessian of the cost function of our problem as stated in~\eqref{eq:prob-vec} is
\begin{equation*}\label{eq:pos-def-hessiana}
\vect{\nabla^2} f(\Tilde{\vect{c}}) = \underbrace{\diag(
    \vect{D}^T(k) \vect{D}(k) \otimes \vect{I}) }_{\vect{\Pi}} +
    \underbrace{\vect{F}^T\vect{\Upsilon}\vect{F}}_{\vect{\Gamma}}
\end{equation*} and we must guarantee that $\vect{\nabla^2} f(\Tilde{\vect{c}}) \succ 0$, which is equivalent to say that the solution exists and is unique. From here on, we will use this knowledge to infer about the conditions on the data.

It is easily seen that the kernel of $\vect{\Gamma}$ is $\text{Ker}(\vect{\Gamma}) = \left. \begin{cases}
    \begin{bmatrix}
        \Bar{\vect{c}} &
        \cdots &
        \Bar{\vect{c}}  
    \end{bmatrix}^T: \Bar{\vect{c}} \text{ constant}\end{cases}\right\}$.
From Lemma~\ref{lemma:a+b}, if we have $\vect{\Pi} + \vect{\Gamma} \succ 0$, then we know that
$ \begin{bmatrix} \Bar{\vect{c}}^T & \cdots & \Bar{\vect{c}}^T \end{bmatrix} \vect{\Pi} \begin{bmatrix} \Bar{\vect{c}}\\ \vdots \\ \Bar{\vect{c}} \end{bmatrix} > \vect{0}$.
Therefore,
\begin{equation*}
\Bar{\vect{c}}^T \left(\sum_{k=1}^{N-1}\vect{D}^T(k) \vect{D}(k) \otimes \vect{I}\right)\Bar{\vect{c}} > 0, \\
\end{equation*}
which entails that
\begin{equation*}
\begin{aligned}
    & \sum_{k=1}^{N-1}\vect{D}^T(k) \vect{D}(k) \succ 0 \\ 
    \Leftrightarrow  & \sum_{k=1}^{N-1}\begin{bmatrix}
    \vect{X}^T (k) & \vect{U}^T (k) \end{bmatrix}^T \begin{bmatrix}
    \vect{X}^T (k) & \vect{U}^T (k) \end{bmatrix} \succ 0 \\ 
    \Leftrightarrow
    & \sum_{k=1}^{N-1}\Big(\begin{bmatrix}
    \vect{x}_1^T (k) & \vect{u}_1^T (k) \end{bmatrix}^T \begin{bmatrix}
    \vect{x}_1^T (k) & \vect{u}_1^T (k) \end{bmatrix} + \\ &\cdots + \begin{bmatrix}
    \vect{x}_\ell^T (k) & \vect{u}_\ell^T (k) \end{bmatrix}^T \begin{bmatrix}
    \vect{x}_\ell^T (k) & \vect{u}_\ell^T (k) \end{bmatrix}\Big) \succ 0 \\ \Leftrightarrow
    & \vect{\Sigma}_1 + \vect{\Sigma}_2 + \cdots + \vect{\Sigma}_L \succ 0.
\end{aligned}
\end{equation*}
Thus, we found the necessary and sufficient condition on data and proved that statements \textit{(i)} and \textit{(ii)} are equivalent.
\end{proof}

Theorem~\ref{theor:unique} allows to efficiently collect data, as we can address the covariance of the data as we collect it and verify its positive definiteness. When this condition is verified, we can stop the data collection and advance for the training.

Equation~\eqref{eq:solution} makes evident that in order to find the optimal value for $\vect{C}$, it is required to invert a matrix 
$$
\vect{\mathcal{A}} := \vect{V}^T \vect{V} +\vect{F}^T \vect{\Upsilon} \vect{F},
$$ $ \vect{\mathcal{A}} \in \Rbb{N(p+q) \times N(p+q)}$. This result shows that the approach requires the computation of a matrix heavily dependent on the number of instants considered in the trajectory being studied and the dimension of the system, meaning the result will become harder to obtain with an increased sampling rate and system complexity. When presented with a high sampling rate system, the size of matrices involved in the calculations can be very high, so solving directly the regularized least squares problem is not feasible. Regardless of this shortcoming, it is still important to understand the conditions in which the matrix $\vect{\mathcal{A}}$ is invertible, since it will impact the numerical solution of the optimization problem. Theorem~\ref{theor:invert} addresses this point.

\begin{theorem}\label{theor:invert}
Matrix $\vect{\mathcal{A}}$ is invertible if there exists a set of $p+q$ pairs $(\ell, k) \in \{1, ..., L\} \times \{0, ..., N-1\}$ such that the corresponding $\begin{bmatrix}
        \vect{x}_\ell^T(k) &
        \vect{u}_\ell^T(k) 
    \end{bmatrix}$ are all linearly independent.
\end{theorem}

\begin{proof}
The proof is made by contraposition, by proving that if the problem does not have a unique solution, then conditions of the theorem do not hold.

Let us consider then that $\vect{\mathcal{A}}$ is not invertible. Then
\begin{equation} \label{eq:invertible-cond}
\begin{aligned}
\underset{\| \bm{\eta} \| = 1}{\exists \bm{\eta}:} \hspace{0.5cm} \bm{\eta}^T \vect{\mathcal{A}} \bm{\eta} = 0 & \Leftrightarrow \bm{\eta}^T \vect{V}^T \vect{V} \bm{\eta} + \bm{\eta}^T \vect{F}^T \vect{\Upsilon} \vect{F} \bm{\eta} = 0 \\
& \Leftrightarrow \| \vect{V} \bm{\eta}\|^2 + \|\vect{\Upsilon}^{\onehalf} \vect{F} \bm{\eta}\|^2 = 0 \\
& \Leftrightarrow \vect{V}\bm{\eta} = \vect{0} \land \vect{\Upsilon}^{\onehalf}\vect{F}\bm{\eta} = \vect{0}
\end{aligned}.
\end{equation}

Defining $\bm{\eta}(k) = \begin{bmatrix}
    \bm{\eta}_A^T(k) & \bm{\eta}_B^T (k)
\end{bmatrix}^T$ and $\bm{\eta} = \left[\bm{\eta}(k)\right]_{k=0}^{k=N-1}$, it is possible to write, for all $k$,
\begin{equation*}
    \begin{aligned}
    & \vect{\Upsilon}^{\onehalf}\vect{F} \bm{\eta} = \vect{0} \Leftrightarrow \\ & \begin{cases}
            \bm{\eta}_A(k) - \bm{\eta}_A(k-1) = \vect{0} \Leftrightarrow \bm{\eta}_A(k) = \bm{\eta}_A(k-1) := \Bar{\bm{\eta}}_A\\
        \bm{\eta}_B(k) - \bm{\eta}_B(k-1) = \vect{0} \Leftrightarrow \bm{\eta}_B(k) = \bm{\eta}_B(k-1) := \Bar{\bm{\eta}}_B
    \end{cases}
    \end{aligned}
\end{equation*} where the fact that $\vect{\Upsilon}^{\onehalf}$ is invertible was used. This yields
\begin{equation*}
    \bm{\eta^*} := \underbrace{\begin{bmatrix}
        \vect{I}\\
        ... \\
        \vect{I}\\
    \end{bmatrix}}_{\vect{\mathcal{N}}} \underbrace{\begin{bmatrix}
            \Bar{\bm{\eta}}_A \\ \Bar{\bm{\eta}}_B
    \end{bmatrix}}_{\Bar{\bm{\eta}}}
\end{equation*} denoting the null space of $\vect{F}$. In order for~\eqref{eq:invertible-cond} to hold, $\bm{\eta^*}$ has to be contained by the null space of $\vect{V}$ whenever $\lambda_k \neq 0$, i.e.
\begin{equation*}
    \vect{V} \bm{\eta^*} = \vect{0} \Leftrightarrow \vect{V} \vect{\mathcal{N}} \Bar{\bm{\eta}} = \vect{0}
\end{equation*} that can be reduced to

\begin{equation}\label{eq:condition}
    \forall_k \forall_{\ell} \hspace{0.5cm} \vect{x}_{\ell}^T(k)  \Bar{\bm{\eta}}_A + \vect{u}_{\ell}^T(k)  \Bar{\bm{\eta}}_B = 0.
\end{equation}

The solution of~\eqref{eq:condition} implies that either $\Bar{\bm{\eta}} = \vect{0}$, contradicting the hypothesis of the proof ($\| \bm{\eta} \| = 1$), or that the vectors $\begin{bmatrix}
        \vect{x}_\ell^T(k) &
        \vect{u}_\ell^T(k) 
    \end{bmatrix}$ in the dataset are linearly dependent, i.e., it is not possible to find $p + q$ linearly independent vectors for all $\ell = 1,...,L$ and $k = 0,...,N-1$. Thus, if $\lambda_k \neq 0$ for all $k$ and the matrix $\vect{\mathcal{A}}$ is singular, the condition of the Theorem cannot hold and, by contraposition, if it holds, then $\vect{\mathcal{A}}$ is invertible and there exists a unique solution to Problem~\ref{prob:RLS}.

This concludes the proof of sufficiency of the conditions to the existence of solution.
\end{proof}

\begin{remark}
The conditions for the existence and uniqueness of solution set forth by Theorem ~\ref{theor:unique} can be clearly seen as equivalent to this new result, as both require the complete set of $\begin{bmatrix}
        \vect{x}_\ell^T(k) &
        \vect{u}_\ell^T(k) 
    \end{bmatrix}$ vectors within the dataset to span $\Rbb{p+q}$.
\end{remark}

Solving~\eqref{eq:solution} is the path adopted by classic solvers to perform its operations, which becomes a disadvantage of this technology, as it is not capable of keeping up with the complexity of the estimation. Hence, addressing this shortcoming is essential to configure the system identification problem of linear time-variant systems as a regularized least squares problem.
\section{Closed-form discrete LTV system identification}

We can regard ~\eqref{eq:prob-compact} for each instant $k$ independently, such that it can be formulated as
\begin{problem}
\begin{equation} \label{eq:prob}
\begin{aligned}
    \underset{\vect{C}}{\text{minimize}}
    \hspace{0.5cm} f(\vect{C}) & := h(\vect{C}) + g(\vect{C})
\end{aligned}
\end{equation}
\end{problem} where $  h(\vect{C}) = \frac{1}{2} \sum_{k=0}^{N-1} \| \vect{D}(k) \vect{C}(k) - \vect{X'}^T(k)\|_F ^2$ and $ g(\vect{C}) = \frac{1}{2} \sum_{k=1}^{N-1} \| \lambda_k^{\onehalf} (\vect{C}(k) - \vect{C}(k-1))\|_F ^2$.

Accordingly, we can think of the derivative of the cost function as a sum of the derivatives of its terms and have for the $k$-th instant 
\begin{equation} \label{eq:gradF-k}
    \vect{\nabla_k} f(\vect{C}) = \vect{\nabla_k}h(\vect{C}) + \vect{\nabla_k}g (\vect{C}),
\end{equation} 
knowing that the gradient at each instant $k$ is $\vect{\nabla_k}f (\vect{C})$. 
The gradient of $h(\vect{C})$ at $k$ is
\begin{equation} \label{eq:grad_h}
    \vect{\nabla_k}h(\vect{C}) = \vect{D}^T(k)(\vect{D}(k) \vect{C}(k) - \vect{X'}^T(k)).
\end{equation}
The gradient of $g(\vect{C})$ is
\begin{equation} \label{eq:grad_g}
    \vect{\nabla}g(\vect{C}) =  \vect{F}^T \vect{\Upsilon} \vect{F}  \vect{C}.
\end{equation}
With $\vect{F}$ as in~\eqref{eq:dif_equation}, we can compute $\vect{F}^T \vect{\Upsilon} \vect{F}$
\begin{equation*}
\resizebox{0.45\textwidth}{!}{$
    \vect{F}^T \vect{\Upsilon} \vect{F} = \begin{bmatrix}
                        \lambda_1 \vect{I} & -\lambda_1 \vect{I} & \vect{0}& ... & \vect{0}& \vect{0}& \vect{0}\\
                        -\lambda_1 \vect{I} & (\lambda_1+\lambda_2)\vect{I} & -\lambda_2\vect{I} & ... & \vect{0}& \vect{0}& \vect{0}\\
                        ... & ... & ... & ... & ... & ... & ...\\
                        \vect{0} &\vect{0}& \vect{0}& ... & -\lambda_{N-2}\vect{I} & (\lambda_{N-2} + \lambda_{N-1})\vect{I} & -\lambda_{N-1}\vect{I} \\
                        \vect{0} &\vect{0}& \vect{0}& ... & \vect{0}& -\lambda_{N-1}\vect{I} & \lambda_{N-1}\vect{I}
                        \end{bmatrix}$}.
\end{equation*}
Replacing~\eqref{eq:grad_h} in~\eqref{eq:gradF-k} and addressing each element of~\eqref{eq:grad_g}, we get $\vect{\nabla_k}f(\vect{C}) = 0$, i.e.,
\begin{equation}\label{eq:SumGradF}
\vect{D}^T(k)(\vect{D}(k) \vect{C}(k) - \vect{X'}^T(k)) + [\vect{F}^T \vect{\Upsilon}\vect{F} \vect{C}](k) = 0.
\end{equation} We have, thus, to solve a linear system of equations. To do so, we leverage on the key observation of Remark~\ref{rem:tridiag}.
\begin{remark}\label{rem:tridiag}
The linear system of equations to be solved in~\eqref{eq:SumGradF} is in fact a block tridiagonal linear system. By considering the $LU$ factorization method, our problem has a closed-form solution with linear complexity~\cite{isaacson_analysis_1994}. 
\end{remark}
Next, we apply the $LU$ factorization to~\eqref{eq:SumGradF}. In general, it consists on a forward pass and then a backward pass, in this case with a complexity that scales linearly with the number of time steps considered in the system model.
Given the previously defined $\vect{D}$, let us have the following auxiliary variables
\begin{equation} \label{eq:aux-defs}
    \begin{cases}
        \vect{S}_{00} = \vect{D}^T(0) \vect{D}(0) + \lambda_1  \vect{I}\\
        \vect{S}_{N-1,N-1} = \vect{D}^T(N-1) \vect{D}(N-1) + \lambda_{N-1} \vect{I} \\
        \vect{S}_{kk} = \vect{D}^T(k) \vect{D}(k) + (\lambda_k + \lambda_{k+1}) \vect{I} \\
        \vect{S}_{k,k-1} = \vect{S}_{k-1,k} = -\lambda_k \vect{I} \\
        \vect{S}_{k,k+1} = -\lambda_{k+1} \vect{I}\\
        \vect{\Theta}_k = \vect{D}^T(k) \vect{X'}^T(k)
    \end{cases}.
\end{equation}
\textbf{Forward pass}

For the forward pass, we solve $\vect{L} \vect{Y} = \vect{\Theta}$, where $\vect{L}$ is
\begin{equation*}
    \vect{L} = \begin{bmatrix}
    \vect{\Lambda}_0 & \vect{0}& \cdots & \vect{0}& \vect{0}\\
    \vect{S}_{1,0} & \vect{\Lambda}_1 & \cdots & \vect{0}& \vect{0}\\
    \vdots & \ddots & \ddots & \vdots & \vdots \\
    \vect{0} &\vect{0}& \cdots & \vect{\Lambda}_{N-2} & \vect{0}\\
    \vect{0} &\vect{0}& \cdots & \vect{S}_{N-1,N-2} & \vect{\Lambda}_{N-1}
    \end{bmatrix},
\end{equation*}
and $\vect{\Lambda}_k$ and $\vect{\Theta}_{k,\ell}$ are square submatrices of size $(p+q)\times (p+q)$.
To start the process, the unknown variables are initialized with
\begin{equation*}
    \begin{cases}
        \vect{\Lambda}_{0} = \vect{S}_{00}\\
        \vect{Y}_0 = \vect{\Lambda}^{-1}_0 \vect{\Theta}_0
    \end{cases},
\end{equation*}
and for the subsequent time steps $k = {1,...,N-1}$, we incrementally compute
\begin{equation*}
    \begin{cases}
        \vect{\Omega}_k = \vect{S}_{k,k-1} \vect{\Lambda}^{-1}_{k-1} \vect{S}_{k-1,k}\\
        \vect{\Lambda}_k = \vect{S}_{kk} - \vect{\Omega}_k\\
        \vect{Y}_k = \vect{\Lambda}^{-1}_k (\vect{\Theta}_k - \vect{S}_{k,k-1} \vect{Y}_{k-1}).
    \end{cases}\label{eq:forw}
\end{equation*}
\textbf{Backward pass}

Given the results of the forward pass, we can address the backward propagation by solving $\vect{M} \vect{C} = \vect{Y}$, with
\begin{equation*}
    \vect{M} = \begin{bmatrix}
    \vect{I}& \vect{\Lambda}^{-1}_0 \vect{S}_{0,1} & \cdots & \vect{0}& \vect{0}\\
    \vect{0} &\vect{I}& \cdots & \vect{0}& \vect{0}\\
    \vdots & \vdots & \ddots & \ddots & \vdots \\
    \vect{0} & \vect{0} & \cdots & \vect{I}& \vect{\Lambda}^{-1}_{N-2} \vect{S}_{N-2,N-1} \\
    \vect{0} & \vect{0} & \cdots & \vect{0}& \vect{I}
    \end{bmatrix}.
\end{equation*}
The optimization variable in the last instant is initialized with $\vect{C}(N-1) = \vect{Y}_{N-1}$. For $i = {N-2,...,0}$,
\begin{equation*}
    \vect{C}(k) = \vect{Y}_k - \vect{\Lambda}^{-1}_k \vect{S}_{k,k+1} \vect{C}(k+1). \label{eq:back}
\end{equation*}
\begin{algorithm}[H]
\caption{COSMIC - closed-form system identification}\label{alg:closed}
\begin{algorithmic}[1]
\Require{$\vect{D}$, $\vect{X'}$, $[\lambda_k]_{k=1}^{k=N-1}$} 
\Ensure{$\vect{C^*}$}

\Statex

\State $\vect{S}_{00} \gets \vect{D}^T(0) \vect{D}(0) + \lambda_1 \vect{I}$
\State $\vect{S}_{N-1,N-1} \gets \vect{D}^T(N-1) \vect{D}(N-1) + \lambda_{N-1} \vect{I}$

\For{$k \in [1,...,N-2]$} 
\State $\vect{S}_{kk} \gets \vect{D}^T (k)\vect{D}(k) + (\lambda_k + \lambda_{k+1}) \vect{I}$
\EndFor

\
\For{$k \in [0,...,N-1]$} 
\State $\vect{\Theta}_k \gets \vect{D}^T(k) \vect{X'}^T(k)$
\EndFor

\
\State $\vect{\Lambda}_{00} \gets \vect{S}_0$
\State $\vect{Y}_0 \gets \vect{\Lambda}_0^{-1} \vect{\Theta}_0$

\
\For {$k \in  [1,...,N-1]$} \Comment{forward pass}
\State $\vect{\Lambda}_{k} \gets \vect{S}_{kk} - \lambda_k^2 \vect{\Lambda}^{-1}_{k-1}$
\State $\vect{Y}_k \gets \vect{\Lambda}^{-1}_k (\vect{\Theta}_k + \lambda_k \vect{Y}_{k-1})$
\EndFor

\
\State $\vect{C}(N-1) \gets \vect{Y}_{N-1}$

\
\For {$k \in  [N-2,...,0]$} \Comment{backward pass}
\State $\vect{C}(k) \gets \vect{Y}_k + \lambda_{k+1} \vect{\Lambda}^{-1}_k \vect{C}(k+1)$ 
\EndFor

\
\Return $\vect{C}^* \gets \vect{C}$

\end{algorithmic}
\end{algorithm}

Next, we state the main result of this work, about the convergence in a fixed, finite number of algebraic operations, i.e., as a closed-form.

\begin{theorem}\label{theor:teorema2}[Algorithm~\ref{alg:closed} solves~\eqref{eq:prob} in closed-form with linear complexity]

Algorithm~\ref{alg:closed} yields a closed-form solution for the system identification problem of the linear time-variant system from data, with linear complexity on the number of time steps.
Namely, the number of multiplication operations to be performed by the Algorithm~\ref{alg:closed} is
\begin{equation*}
    c = N \left ( (p+q)^3 + (2p + 3)(p+q)^2 \right),
\end{equation*}
which is linear on the number of time steps $N$, and cubic on the size of state space and control space, and does not depend on $L$, the size of the dataset used for learning matrices $[\vect{A}(k), \vect{B}(k)]_{k=0}^{N-1}$.
\end{theorem}

\begin{proof}
We must prove that solving $\nabla f(\vect{C}) = 0$ is equivalent to solving
\begin{equation}\label{eq:lmc=q}
    \vect{L} \vect{M} \vect{C} = \vect{\Theta},
\end{equation}
where $\vect{L}$ is a lower triangular matrix with non zero diagonal and first subdiagonal blocks, and $\vect{M}$ is an upper triangular matrix with diagonal identity submatrices and first superdiagonal non-zero blocks. 

Let us assume a system starting from the results in~\eqref{eq:SumGradF}. We can rearrange the system of linear equations to address it as
\begin{equation*}
\begin{aligned}
    & \resizebox{0.45\textwidth}{!}{$\underbrace{\begin{bmatrix}
    \vect{D}^T(0) \vect{D}(0) \vect{C}(0) + \lambda (-\vect{C}(0) + \vect{C}(1))\\
    \vdots \\
    \vect{D}^T(k) \vect{D}(k) \vect{C}(k) + \lambda (-\vect{C}(k-1) + 2\vect{C}(k) - \vect{C}(k+1))\\
    \vdots \\
    \vect{D}^T(N-1) \vect{D}(N-1) \vect{C}(N-1) + \lambda (-\vect{C}(N-2) + \vect{C}(N-1)
    \end{bmatrix}}_{\vect{LMC}}$} \\ & =    
    \underbrace{\begin{bmatrix} 
    \vect{D}^T(0) \vect{X'}^T(0) \\
    \vdots \\
    \vect{D}^T(k) \vect{X'}^T(k) \\
    \vdots\\
    \vect{D}^T(N-1) \vect{X'}^T(N-1)
    \end{bmatrix}}_{\vect{\Theta}},
\end{aligned}
\end{equation*}
which directly yields the definition of $\vect{\Theta}$, just like it was defined in~\eqref{eq:aux-defs}. Decomposing the LHS of the previous definition, we can identify the dependence on $\vect{C}$ and single it out. Thus, we get
\begin{equation} \label{eq:defLM}
\underbrace{\begin{bmatrix}
    \vect{S}_{00} & - \lambda_1 \vect{I} & \cdots & \vect{0}\\
    -\lambda_{1}\vect{I} &  \vect{S}_{11} & \ddots & \vdots \\
    \vdots & \ddots  & \ddots &-\lambda_{N-1} \vect{I}  \\
    \vect{0} & \cdots & -\lambda_{N-1} \vect{I} & \vect{S}_{N-1,N-1}
    \end{bmatrix}}_{\vect{L} \vect{M}}
    \begin{bmatrix} 
    \vect{C}(0)\\
    \vdots\\
    \vect{C}(N-1)
    \end{bmatrix}.
\end{equation}
Additionally, from the assumption of the form of matrices $\vect{L}$ and $\vect{M}$, as previously defined, we get $\vect{L} \vect{M} =$ 
\begin{equation*} \label{eq:otherDefLM}
\begin{bmatrix}
    \vect{\Lambda}_0  & \vect{S}_{01} & \vect{0} & \vect{0}\\
    \vect{S}_{01} & \vect{\Lambda}_1 + \vect{\Omega}_1 & \ddots & \vdots\\
    \vdots &  \ddots & \ddots & \vect{S}_{N-2,N-1}\\
    \vect{0} & \vect{0} & \vect{S}_{N-1,N-2} & \vect{\Lambda}_{N-1} + \vect{\Omega}_{N-1}
\end{bmatrix}
\end{equation*}
yielding the results we stated in the formulation of the problem by making it equal to~\eqref{eq:defLM}. Hence, we prove that the problem can be solved by a $LU$ factorization and that solving equation~\eqref{eq:cond}, taking into account the nuances of the $k$-th instants, is equivalent to solving ~\eqref{eq:lmc=q}.
For the complexity result, we analyze the multiplication operations involved in the presented algorithm as follows.

For the \emph{forward pass} we need to invert all $N$ $\vect{\Lambda}_k$ matrices. If we use traditional inversion methods, the number of operations for each inversion is $(p+q)^3$. Multiplications by a scalar involve $2(p+q)^2$. The multiplication of the two matrices in Step~11 of Algorithm~\ref{alg:closed}, $\vect{\Lambda}_k^{-1}$, with a matrix $(p+q) \times p$ costs $p(p+q)^2$, gives a total number of operations for all $N$ steps of
$$
    c_\text{fwd} = N\left ((p+q)^3 + (p+2)(p+q)^2  \right ),
$$
while the \emph{backward pass} demands multiplication of a scalar by a matrix of size $(p + q) \times (p + q)$  accounting for $(p+q)^2$ operations, and a multiplication between a matrix of size  $(p + q) \times (p + q)$, and one of size  $(p + q) \times p$, yielding $p(p+q)^2$ operations. Thus, the total number of backward pass operations is 
$$
    c_\text{bwd} = N \left((p + 1) (p+q)^2\right).
$$ \end{proof}
\vspace{-0.5cm}
\subsection{Preconditioning}
To account for eventual irregularities in the data collected that may impair the mathematical operations needed to compute a solution, namely the multiple matrices that need to be inverted, we perform a preconditioning process when necessary. The proposed procedure can be compacted by a new definition of matrices $\vect{S}$ and $\vect{\Theta}$ that regulate the formulation of the upper and lower triangular matrices~\cite{axelsson1996iterative}. As such, indicated with the superscript $\vect{PC}$, we redefine~\eqref{eq:aux-defs} and get
\begin{equation*} \label{eq:aux-precond}
    \begin{cases}
        \vect{S^{PC}}_{kk} =\vect{I}\\
        \vect{S^{PC}}_{i,j} = \vect{S}^{-1}_{ii} \vect{S}_{i,j} & \text{for } i \neq j \\
        \vect{\Theta^{PC}}_k = \vect{S}^{-1}_{kk} \vect{\Theta}_k
    \end{cases}.
\end{equation*} Applying this new formulation to Algorithm~\ref{alg:closed}, we get a more robust solution to the problem, circumventing numerical destabilization caused by large matrix condition numbers. For our simulations, we found that it was not necessary to apply preconditioning to the data. We note that applying preconditioning will change the result of Theorem~\ref{theor:teorema2}, although not changing the linear complexity in $N$.

\subsection{COSMIC validation}

To demonstrate the previously stated theoretical results, a simulation and testing environment was developed.
A classical spring-mass-damper system is first simulated and excited in a Simulink model to allow for data collection both in LTI and LTV setups. The validation is performed with parameter $\lambda_k$ constant throughout the trajectory, thus being represented by $\lambda$. 

This data is used as input into COSMIC and multiple models are trained, in a Python environment, to understand the behavior of the proposed algorithm in different settings, by either varying the $\lambda$ or the noise that disturbs the state.

Finally, the estimated model is excited with previously selected inputs and we are able to evaluate the performance, by comparing with the ground truth. The metrics used are the estimation error and the prediction error.

Taking into account the spring-mass-damper system, considered the ground truth, the estimation error, i.e. how close the estimated model, $\hat{\vect{C}}$, is to the true system, is defined as $\| \hat{\vect{C}} - \vect{C_{gt}}\|_F$. Regarding the predictive power of COSMIC, we propose the evaluation of the predicted state, $\hat{\vect{x}}(k+1)$ when compared to the true state from a simulated trajectory that was not part of the training data set. As such, we address the metric by defining the norm of the error as $ \| \hat{\vect{x}}(k+1) - \vect{x}(k+1) \|_2$, allowing for the assessment of the error propagation. Moreover, for the following validation study, we assumed that the measurements were disturbed by noise, which can be characterized by $\omega \sim \vect{\mathcal{N}}(0,\,\sigma^{2})$ with variance $\sigma^2$.

\begin{figure}[!htb]
  \centering
  \includegraphics[width=0.35\textwidth]{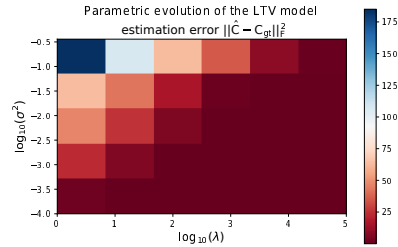}
  \caption[Comparative visualization of the estimation error.]{System estimation performance for different noise levels and values of $\lambda$, compared with the true system. Higher values for the noise and lower values of $\lambda$ have a negative impact in the system performance. The redder in the color scale representing the error, the better.}
  \label{fig:parametric-ltv-comp}
\end{figure}

To infer how the performance of the algorithm depends on the measurement noise and parameter $\lambda$, we opted to develop a parametric study where, for each noise variance, we tested the estimation for a reasonable range of values for parameter $\lambda$. Figure~\ref{fig:parametric-ltv-comp} exposes the variation of the estimation error with the noise added to the measurements for the LTV system and it is clear that for a noise standard deviation of about 10\% of the maximum initial conditions, the algorithm is heavily affected. Reducing this value, the estimation error drops significantly. It is evident that smaller values for the parameter $\lambda$ lead to a worse performance of the algorithm, as the variation between instants is not significant, resulting much better with higher values of $\lambda$.

As part of the validation process, we can also evaluate what is the $\lambda$ that best adapts the estimated model to the true system characteristics. To do so, we evaluate the system in the previous conditions, choosing an appropriate value for the noise, in this case a standard deviation of about 1\% of the initial value. The $\lambda$ that best fits the system we are trying to estimate is $\lambda = 10^5$, as the lowest errors occur at this value.


\subsection{Benchmark}

The novelty of the presented algorithm comes from its efficiency, mainly when compared to other approaches to the same problem. As we are solving a convex optimization problem, we can obtain a solution through \texttt{cvxpy}, a domain specific language for convex optimization problems and is widely used ~\cite{diamond2016cvxpy}. For this specific use case, the ECOS solver~\cite{domahidi2013ecos}, an interior-point solver for second-order cone programming, is automatically invoked. Hence, we evaluate the problem agnostic approach, as in~\eqref{eq:solution} with this strategy. 

\begin{algorithm}
\caption{Stochastic Block Coordinate Descent}\label{alg:SBCD}
\begin{algorithmic}[1]
\Require{$\vect{D}$, $\vect{X'}$, $[\lambda_k]_{k=1}^{k=N-1}$, $\varepsilon$} 
\Ensure{$\vect{C^*}$}

\Statex

\State $\vect{C_0} \gets$ initialized with random values

\State $\vect{S}_{00} \gets \vect{D}^T(0) \vect{D}(0) + \lambda_1 \vect{I}$
\State $\vect{S}_{N-1,N-1} \gets \vect{D}^T(N-1) \vect{D}(N-1) + \lambda_{N-1} \vect{I}$

\For{$k \in [1,...,N-2]$} 
\State $\vect{S}_{kk} \gets \vect{D}^T (k)\vect{D}(k) + (\lambda_k + \lambda_{k+1}) \vect{I}$
\EndFor

\
\For{$k \in [0,...,N-1]$} 
\State $\vect{\Theta}_k \gets \vect{D}^T(k) \vect{X'}^T(k)$
\State $\vect{\nabla h_0}(k)\gets \vect{D}^T(k) (\vect{D}(k)^T \vect{C}_0(k) - \vect{X'}^T(k)) $
\EndFor

\
\State $\vect{\nabla} g_0 \gets \vect{F}^T \vect{\Upsilon}\vect{F} \vect{C}_0$
\State $\vect{\nabla} f_0 \gets \vect{\nabla}h_0 +\vect{\nabla}g_0$
\State $t \gets 0$

\
\While{$\|\vect{\nabla}f (\vect{C}_t) \|_F^2 > \varepsilon $}

\State $RndOrd = random(0, ..., N-1)$
\For{$i \in RndOrd$}

        \If{$i$ is 0}
            \State $\vect{C}_{t+1}(i) \gets \vect{\Theta}_{ii}^{-1} (\vect{\Theta}_i + \lambda_{i+1} \vect{C}_t (i+1))$
        \ElsIf{$i$ is $N-1$}
            \State $\vect{C}_{t+1}(i) \gets \vect{\Theta}_{ii}^{-1} (\vect{\Theta}_i+ \lambda_i \vect{C}_t (i-1))$
        \Else
            \State $\vect{C}_{t+1}(i) \gets \vect{\Theta}_{ii}^{-1} (\vect{\Theta}_i+ \lambda_i \vect{C}_t (i-1) + \lambda_{i+1} \vect{C}_t (i+1))$
        \EndIf

        \
        \State $\vect{\vect{\delta}} (i) \gets \vect{C}_{t+1}(i) - \vect{C}_{t}(i)$
        
        \
        \State $ \vect{\nabla}h_{t+1}(i) \gets  \vect{\nabla}h_{t} + \vect{D}^T(i)\vect{D}(i) \vect{\vect{\delta}}(i)$
        
        \If{$i$ is 0}
            \State $ \vect{\nabla}g_{t+1}(i) \gets \vect{\nabla}g_{t}(i)  + \lambda_{i+1} \vect{\delta}(i)$
            \State $ \vect{\nabla}g_{t+1}(i+1) \gets \vect{\nabla}g_{t}(i+1)  - \lambda_{i+1} \vect{\delta}(i)$
        \ElsIf{$i$ is $N-1$}
            \State $ \vect{\nabla}g_{t+1}(i-1) \gets \vect{\nabla}g_{t}(i-1)  - \lambda_{i} \vect{\delta}(i)$
            \State $ \vect{\nabla}g_{t+1}(i) \gets \vect{\nabla}g_{t}(i)  + \lambda_{i} \vect{\delta}(i)$
        \Else
            \State $ \vect{\nabla}g_{t+1}(i-1) \gets \vect{\nabla}g_{t}(i-1)  - \lambda_{i} \vect{\delta}(i)$
            \State $ \vect{\nabla}g_{t+1}(i) \gets \vect{\nabla}g_{t}(i)  +(\lambda_{i} + \lambda_{i+1})\vect{\delta}(i)$
            \State $ \vect{\nabla}g_{t+1}(i+1) \gets \vect{\nabla}g_{t}(i+1)  - \lambda_{i+1} \vect{\delta}(i)$
        \EndIf
        
        \
        \For{$i \in [i-1, i, i+1]$ }
            \State $\vect{\nabla}f_{t+1} \gets \vect{\nabla}h_{t+1} + \vect{\nabla}g_{t+1} $ \Comment{Only the necessary}
        \EndFor
        \State $t \gets t + 1$
\EndFor
\EndWhile

\Return $\vect{C}^* \gets \vect{C}_t$

\end{algorithmic}
\end{algorithm}

In addition, we develop a solution of~\eqref{eq:prob} through a Stochastic Block Coordinate Descent (SBCD) algorithm.  This approach allows for a simplification regarding \texttt{cvxpy}, as it treats each $\vect{C}(k)$ as an independent optimization variable for an iteration of the algorithm, while keeping all the other values constant, simplifying the calculations and allowing for an easier derivative verification. By evaluating the value of the derivative at the new $\vect{C}$, we can infer if the optimality condition was reached or not. However, this approach is iterative, which means that it may not be able to reach a solution if the proper limits are not met. Nevertheless, the SBCD approach enables different loss functions for the data fidelity term, like, e.g., the Huber M-estimator, robust to outlier measurements. Algorithm~\ref{alg:SBCD} describes in detail the implementation of SBCD for our optimization problem.

Finally, we run again the COSMIC algorithm. These simulations, as all the ones run in a Python environment, were made in Google Colab development setting, with 13GB of RAM available and the Intel(R) Xeon(R) CPU @ 2.20GHz processor.

\begin{table}[H]
\caption{Performance comparison of different optimization strategies to solve Problem~\ref{prob:prob-with-sum}. Time in seconds.}
\vspace{0.2cm}
\begin{center}
\begin{tabular}{lcccc}
\hline
\textbf{Instants} & \textbf{\texttt{cvxpy}} & \textbf{\textit{SBCD}}& \textbf{\textit{COSMIC}} \\
 &  \textbf{Time} & \textbf{Time} &  \textbf{Time}\\
\hline
100 & 2.929 & 0.031 & \textbf{0.014}\\
1000 & 42.588 & 0.273 & \textbf{0.106}\\
10000 & * & 3.531 & \textbf{1.017}\\
100000 & * & 21.055 & \textbf{9.696}\\
\hline
\multicolumn{4}{l}{\small *Session crashed after using all available RAM.} 
\end{tabular}
\label{table:compare}
\end{center}
\end{table}
To perform the comparison in Table~\ref{table:compare}, as we are studying the performance of the methods in cost and time, we simulated a 10 Hz sampling rate trajectory and varied its length between 10 s, 100 s, 1000 s and 10000 s with $\lambda = 10^{-3}$ and stopping the SBCD at 1 million iterations, while using the setting presented in Algorithm~\ref{alg:SBCD}.

The \texttt{cvxpy} approach tries to solve the problem for all time steps simultaneously, leading to a solution that requires the machine to solve a convex problem involving a $N(p+q) \times N(p+q)$ matrix, as stated in~\eqref{eq:solution}. As the number of instances per trajectory increases, the complexity also increases and the \texttt{cvxpy} approach can no longer compute a solution, while both SBCD and closed-form COSMIC approaches continue to perform as expected. Regarding the cost function, no significant changes are observed from one method to another, as expected.

This leads to the conclusion that the closed-form algorithm is the best option to solve~\eqref{eq:prob}, as it can achieve minimal cost values with an efficient use of computational resources, reaching a solution for all cases tested, performing significantly better than the other approaches, reducing up to three of orders of magnitude the computational time from the ECOS solver. Moreover, analyzing Table~\ref{table:compare}, we can experimentally observe the linearity with the number of instances per trajectory that is stated in Theorem~\ref{theor:teorema2}.

\section{Controller design for the estimated model}
\label{chapter:results}







Regarding the estimated model, using the same system and simulation setting that was used for the validation procedure, we applied the parameters found to be the most suiting to develop and save a model, i.e. $\lambda = 10^5$, and the measurements were disturbed by noise with standard deviation $\sigma = 0.06$. Then, we load the new model into another Python environment to allow for it to be used in the dynamic programming algorithm that results in an optimal control time varying gain. Afterwards, the solution is returned to Simulink and tested in the original spring-mass-damper system.

For the design phase of the controller, we develop a tuning framework taking into account the LQR needs, keeping matrices $\vect{Q}_k$ and $\vect{R}_k$ constants throughout the trajectory and depending only on the elements from their diagonals, entailing that the most suiting design parameters are $\frac{q_v}{q_x} = 10^{-1}$, $\frac{r}{q_x} = 10^{-3}$ and $q_x = 1$. Figure~\ref{fig:init-cond} shows the system response to different inputs when controlled by the LTV LQR drawn from the estimated model. It is clear that the controller designed from the COSMIC estimated model can follow the reference from different initial conditions and it is a good solution to the control challenge. 

\begin{figure}[htb]
  \centering
  \includegraphics[width=0.3\textwidth]{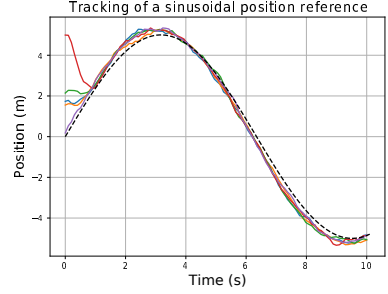}
  \caption{System position response from different initial conditions. The controller designed is able to follow the reference and respond quickly to different initial conditions.}
  \label{fig:init-cond}
\end{figure}

After performing a statistical analysis of the tracking error, we can confidently say that the estimated model is a good approximation of the system, as shown by the close values of the statistical metrics in Table~\ref{table:stats-control-error}. Moreover, it allows for a better controller synthesis than previously used techniques, performing better than the LTI LQR, presenting a smaller mean and sum of squared error, once again from Table~\ref{table:stats-control-error}. 

Thus, the model estimated by COSMIC is a good approximation of the spring-mass-damper system used throughout this work and the dynamic programming strategy worked well for the problem posed. Hence, this system identification and controller design framework is validated and we can further test it in more complex environments.

\begin{table}[ht!]
\caption{Statistical comparison of the controller design methods. Mean, standard deviation and sum of the squared errors relative to the reference tracking errors for the optimal controllers considered.}
\vspace{0.1cm}
\begin{center}
\begin{tabular}{lccc}
\hline
\textbf{Controller} & \textbf{Mean} & $\mathbf{\sigma}$ & $\mathbf{\sum \text{error}^2}$ \\
\hline
\textbf{COSMIC} & \textbf{0.1050} & \textbf{0.7346} & \textbf{0.5506} \\
Ground truth & 0.0994 &  0.7441 & 0.5635\\
Time-invariant & 0.1218 & 0.7425 & 0.5661\\
\hline
\end{tabular}
\label{table:stats-control-error}
\end{center}
\end{table}

\section{Comet Interceptor as a Case Study}
Using the framework formulated throughout this work in the Comet Interceptor mission, we can verify its applicability to real world systems. The work presented in this section is an approach to the first step of the controller design process, where the nonlinear system is represented in a Simulink model, exploiting the physical knowledge of the spacecraft attitude. Firstly, we use a Low Fidelity Simulator, where only the nonlinear dynamics of the problem are modeled. Furthering the analysis, taking into account external perturbations, we use a Functional Engineering Simulator to collect data, which is more demanding. System identification results evaluation is done by comparing the estimated model with the linearized and discretized attitude error system, as described in~\cite{lourenco:2021:ModelPredictiveControl}.

We opt to address the problem with varying $\lambda_k$. Taking into account the previous knowledge about the system, i.e. the angular velocity over the $y$ axis that is desired, we establish that the work can be performed considering two different values for $\lambda_k$ and three zones (the middle of the trajectory and the extremities) where these values will be maintained constant. 

\begin{figure}[!htb]
    \centering
    \includegraphics[width=0.45\linewidth]{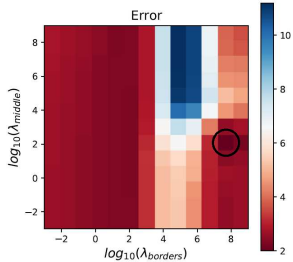}
  \caption{Parametric results for the variation of $\lambda_k$ along the trajectory in the Low Fidelity Simulator. The best results occur for $\lambda_{middle} = 10^2$ and $\lambda_{borders} = 10^8$, indicated by the black circle. The redder in the color scale representing the error, the better.}
  \label{fig:lambda-comp-CISim}
\end{figure}

Regarding the Low Fidelity Simulator, we performed a parametric study of the state error, varying both $\lambda_{middle}$ and $\lambda_{borders}$. Analyzing Figure~\ref{fig:lambda-comp-CISim}, we choose $\lambda_{middle} = 10^2$ and $\lambda_{borders} = 10^8$, which present the lowest state estimation errors. Although the neighboring values for the regularization factor result in bigger errors, we verified in the simulations that this was not a problem for our solution. The chosen values for $\lambda_k$ are in line with the analysis of the system and its expected behavior. For the zones further away from the closest approach point, the system dynamics is much slower and the variations between instants are less disruptive, which is in line with a higher value of $\lambda_k$, that imposes a narrower difference between the optimal variable variation in consecutive instants. In the middle of the trajectory, the spacecraft is at its closest point to the target and its attitude needs to change much faster to maintain a pointing error small enough, thus the difference between two instants needs to be larger, which is allowed by the smaller value of $\lambda_{middle}$.

\begin{figure}[!htb]
\begin{subfigmatrix}{2}
    \subfigure[System identification for the LFS, represented by one of the state components, the angular error in the $z$ axis. The COSMIC estimated system, in red, always coincides with the true system, where the linearization, in blue, does not.]
    {\centering\includegraphics[width=0.56\linewidth]{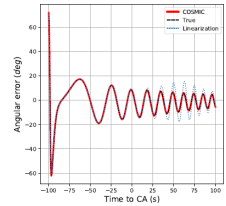}}
    \subfigure[Pointing error resulting from the designed controller application to LFS. Each line is a trajectory with different initial conditions.]
    {\includegraphics[width=0.56\linewidth]{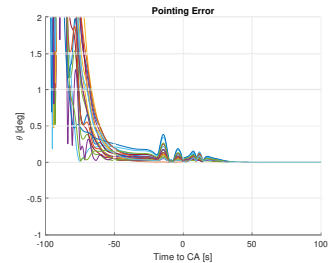}}
\end{subfigmatrix}
\caption{Results for the Low Fidelity Simulator.}
\label{fig:lin-cosmic-comp}
\end{figure}

The results shown in Figure~\ref{fig:lin-cosmic-comp}(a) refer to the implementation of COSMIC with data from the Low Fidelity Simulator. We verify that the results from COSMIC are closer to the real system than the results from the linearized system, which can be seen in the difference in the oscillatory behavior after the point of closest approach, with the blue line moving away from the ideal result. As such, we also developed a new controller, with the dynamic programming framework, and we are now able to input the optimal gains calculated for each distinct instant into the simulator, instead of the constant gain that had been previously used. Figure~\ref{fig:lin-cosmic-comp}(b) shows the results from varying initial conditions and it is clear that the controller from the linear system identification performed by COSMIC is able to control the nonlinear system and is a good option for the first approach to the GNC framework of the mission.

Moving on to the Functional Engineering Simulator, where effects such as comet particles impacts are modeled, we present the results of the system identification in Figure~\ref{fig:traj-eval}. Once again, the COSMIC model is much closer to the true system than the linearized error dynamics. In this case, we can observe oscillations in the linearized system that do not exist in the COSMIC estimated case. Therefore, we can infer that COSMIC can identify the system even with the external disturbances and, once again, that it can perform better than the linearization approach.



\begin{figure}[!htb]
\centering
\includegraphics[width=0.45\linewidth]{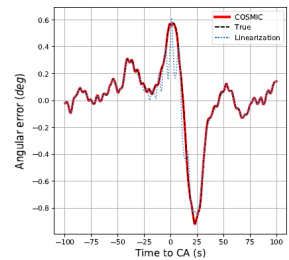}
\caption{System identification results for the Functional Engineering Simulator, represented by one of the state components, the angular error in the $z$ axis. COSMIC estimated system, red, is much closer to the true system than the linearization, blue.}
\label{fig:traj-eval}
\end{figure}

\section{Conclusion}
The main contribution of this work was the development of COSMIC, a closed-form system identification algorithm from data for linear time-variant systems, formulating the identification problem as a regularized least squares problem, with a regularization term that discourages a high frequency variation between the solution consecutive instances. We put forward COSMIC, a closed-form solution method, that scales linearly with the number of time instants considered. COSMIC, as a closed-form algorithm, converges after a finite and well-known number of steps, thus removing the uncertainty and additional work of deciding on a stopping criterion. COSMIC reduces in three orders of magnitude of the computation time when compared to a general purpose solver and it is faster than a special designed iterative approach. We are confident that the results from this work will help develop the system identification procedure for LTV systems and the representation as a regularized least squares problem can be extremely useful for multiple settings. Moreover, the application in a real Space mission is a great step towards taking advantage of all the data available and we COSMIC is used in the future as a tool to expedite the early phases of development of the GNC and AOCS frameworks for different missions.

\begin{ack}
We would like to thank João Franco, from CoLAB+ Atlantic collocated at GMV, for the support provided and his work on the simulator development.  
\end{ack}

\bibliographystyle{plain}        
\bibliography{autosam}           
\end{document}